\documentclass[12pt,a4paper]{amsart}
\usepackage{amsmath, amssymb, amsfonts, amsthm, float, stmaryrd, epsfig}
\usepackage[abbrev,alphabetic]{amsrefs}
\usepackage[pdfborder={0 0 0 [0 0 ]},bookmarksdepth=3, bookmarksopen=true]{hyperref}
\usepackage[latin1,utf8]{inputenc}
\usepackage[T1]{fontenc}
\usepackage{microtype}
\usepackage[cmintegrals,cmbraces]{newtxmath}
\usepackage{ebgaramond-maths}

\usepackage[capitalize]{cleveref}
\usepackage{color}
\usepackage[ps,all,arc,rotate]{xy}
\usepackage{bbm} 

\BibSpecAlias{misc}{webpage}

\usepackage{stackengine,scalerel}
\stackMath

\usepackage{booktabs}

\hypersetup{
    colorlinks=true,
    linkcolor=black,
    citecolor=black,
    filecolor=black,
    urlcolor=black,
}

\numberwithin{figure}{section}
\numberwithin{table}{section}

\theoremstyle{plain}
\newtheorem{thm}{Theorem}[section]
\crefname{thm}{Theorem}{Theorems}
\newtheorem*{prop*}{Proposition}
\newtheorem*{thm*}{Theorem}

\crefname{prop}{Proposition}{Propositions}
\newtheorem{lem}[thm]{Lemma}
\crefname{lem}{Lemma}{Lemmata}
\newtheorem{cor}[thm]{Corollary}
\crefname{cor}{Corollary}{Corollaries}

\crefname{conj}{Conjecture}{Conjectures}
\crefname{equation}{Equation}{Equations}

\newtheorem{thmx}{Theorem}
\crefname{thmx}{Theorem}{Theorems}

\theoremstyle{definition}
\newtheorem{ex}[thm]{Example}

\newtheorem{dfn}[thm]{Definition}
\newtheorem*{dfn*}{Definition}

\theoremstyle{remark}
\newtheorem{rmk}[thm]{Remark}

\newtheoremstyle{maintheorem}{}{}{\itshape}{}{\bfseries}{}{.5em}{#1 \!\thmnote{\ #3}.}
\theoremstyle{maintheorem}

\makeatletter
\let\c@figure\c@thm
\let\c@table\c@thm
\makeatother
\crefname{figure}{Figure}{Figures}
\crefname{table}{Table}{Tables}

\newcommand{\Aut}{\operatorname{Aut}}
\newcommand{\SAut}{\operatorname{SAut}}

\newcommand{\GL}{\operatorname{GL}}
\newcommand{\SL}{\operatorname{SL}}

\newcommand{\MCG}{\operatorname{MCG}}

\newcommand{\Sp}{\operatorname{Sp}}

\newcommand{\I}{\mathrm{I}}

\def\C{\mathbb{C}}
\def\R{\mathbb{R}}

\def\Z{\mathbb{Z}}

\def\1{\mathbbm{1}}

\def\s-{\smallsetminus}

\def\iff{if and only if }

\newcommand{\Adj}{\operatorname{Adj}}
\newcommand{\Sq}{\operatorname{Sq}}

\newcommand{\type}[2]{\mathtt{#1}_{\mathtt{#2}}}

\newcommand{\typeA}[1]{\type{A}{#1}}
\newcommand{\typeB}[1]{\type{B}{#1}}
\newcommand{\typeC}[1]{\type{C}{#1}}
\newcommand{\typeD}[1]{\type{D}{#1}}
\newcommand{\typeE}[1]{\type{E}{#1}}
\newcommand{\typeF}[1]{\type{F}{#1}}
\newcommand{\typeG}[1]{\type{G}{#1}}

\newcommand{\A}{\mathcal{A}}
\newcommand{\Airr}{\A_{\textsl{irr}}}

\usepackage{xfp}
\def\lambdaSlthreea{0.280406}
\def\lambdaSlthree{0.542497}
\def\lambdaSlfour{1.316499}
\def\lambdaSlfive{2.690925}
\def\lambdaSpfourb{0.879159}
\def\lambdaSpfourc{1.412187}
\def\lambdaAdjAbb{0.158606}
\def\lambdaAdjAbc{0.273954}
\def\lambdaAdjCbc{0.244935}
\def\lambdaAdjGbc{1.56799}
\def\lambdaGbb{0.967685}
\def\lambdaLevels{2.417393}

\newcommand{\Kzhconst}[2]{\fpeval{floor(sqrt(2*#1/#2), 5)}}

\DeclareMathOperator{\lspan}{Span}
\DeclareMathOperator{\lev}{Lev}

\binoppenalty=\maxdimen
\relpenalty=\maxdimen

\newcounter{dawidcomments}
\newcommand{\dawid}[1]{\textbf{\color{red}(D\arabic{dawidcomments})}
\marginpar{\scriptsize\raggedright\textbf{\color{red}(D\arabic{dawidcomments})Dawid: }#1}
\addtocounter{dawidcomments}{1}}

\newcounter{marekcomments}

\author{Marek Kaluba}
\email{marek.kaluba@kit.edu}
\address[M.~Kaluba]{Department of Mathematics
    Karlsruhe Institute of Technology
    Englerstr.~2
    D-76131 Karlsruhe
    Germany
}
\author{Dawid Kielak}
\email{kielak@maths.ox.ac.uk}
\address[D.~Kielak]{Mathematical Institute,
	Andrew Wiles Building,
	Observatory Quarter,
	University of Oxford,
	Oxford,
	OX2 6GG,
	United Kingdom}
\title[Kazhdan constants for Chevalley groups over $\Z$]{Kazhdan constants for Chevalley groups\\ over the integers}

\begin{document}

\begin{abstract}
    We compute lower bounds for Kazhdan constants of Chevalley groups over the integers,
    endowed with the standard Steinberg generators.
    For types other than $\typeA{n}$, these are the first explicit asymptotically sharp such bounds.
    The method relies on establishing a new connection between the
    structure of a root system grading a family of groups and
    the behaviour of the square of the Laplace operator in the family.
    %
\end{abstract}

\maketitle

\section{Introduction}

In his seminal paper \cite{Kazhdan1967}, Kazhdan introduced property $T$, an important rigidity property for locally compact groups.
For such a group $G$ the definition stipulates the existence
of a compact subset $S \subset G$ and a constant $\kappa$,
such that for every unitary representation $\pi$ of $G$ either
every element of $S$ moves all unit vectors by at least $\kappa$
or $\pi$ contains a non-zero $G$-invariant vector.
When $G$ is discrete, the set $S$ can be easily seen to be necessarily generating and
the maximal  $\kappa = \kappa(G,S)$ as above is known as the \emph{Kazhdan constant}.

In the paper Kazhdan proved that all simple algebraic groups over
local fields of rank at least $3$, and most such groups of rank $2$, have property $T$;
his result also covered lattices in such groups.
This was later extended by Vaserstein~\cite{Vaserstein1968} to
cover all groups of rank at least $2$.
In particular, using the theorem of Borel--Harish-Chandra~\cite{BorelHarish-Chandra1962}
we see that the result applies to Chevalley groups
over the integers associated to any irreducible root system of rank at least $2$.

It turned out that for many
applications, not limited to the estimates of the expansion constants,
one needs the quantitative knowledge of the associated Kazhdan constants
(see e.g. the introduction of \cite{BekkaHarpeValette} for an overview).
Unfortunately, neither of the results quoted above sheds light on the possible values of 
such constants.

The computation of these values is a problem with a long and distinguished history,
discussed for example by de la Harpe and Valette in \cite[Chapter 1]{delaHarpeValette1989}.
In particular, they restate the question of Serre about the Kazhdan constant
for $\SL_3(\Z)$ endowed with the standard generating set.

Finding Kazhdan constants, somewhat surprisingly, seems much harder than 
establishing property $T$.
For finite groups Kazhdan constants can be computed since the
representation theory of such groups is completely understood, see for example \cite{BacherDelaharpe1994}.
To the authors' knowledge the only cases of \emph{infinite} groups where
the exact values are known are groups acting transitively on $\widetilde{A}_2$ buildings,
as computed by Cartwright--M{\l}otkowski--Steger \cite{Cartwright1994}.
However, the problem of establishing a lower bound on the constants seems slightly
more accessible and has been an area of very active research recently.

There are several methods that can be used to bound Kazhdan constants from below.
Broadly speaking, one can:
analyse angles between subspaces in unitary representations \cites{Burger1991,DymaraJanuszkiewicz2002};
use the property of bounded generation \cite{Shalom1999,Neuhauser2003,Kassabov2005};
establish large spectral gap of the graph Laplacian of a link in a Cayley graph \cite{DymaraJanuszkiewicz2002,Zuk2003};
present the group as a graph of groups with sufficient co-distance at vertices
\cite{ErshovJaikin2010,Ershovetal2017};
or use sums-of-squares decompositions in the group ring \cite{Ozawa2016}.
We are going to focus on this last method.
A different, geometric method of partially defined cocycles
\cite[Section 4]{Zuk2003} seems to be under-explored so far
from the theoretical point of view (cf. \cite{Nitsche2020,Ozawa2022}).

The mentioned method of Ozawa focuses on finding a decomposition of
the squared group Laplacian into sum of squares.
While the advantage of this formulation is its simplicity,
the decomposition must be repeated for every group, separately,
unless the family of groups which we are interested in has some internal structure.
E.g. \cite{Kalubaetal2021} leverages the structure of
$\{\SL_n(\Z)\}_{n\geqslant 3}$ to exhibit a \emph{single computation}
that proves property T for \emph{all} groups in the family.
The external similarities between the groups $\SL_n(\Z)$ and $\Aut(F_n)$
(for which the method of \cite{Kalubaetal2021} also works)
lead the authors to believe that these results are a specialization of a more
general pattern which is independent of the nature of the particular groups
and depends on the internal structure of the family instead.

\bigskip
The central aim of this paper is to generalize the method of \cite{Kalubaetal2021}
beyond the family of special automorphisms groups of $\Z^n$ (or $F_n$)
and establish a new general method for computing lower bounds for
Kazhdan constants across families of groups.
In the process the authors found that the method is particularly well tailored
to the families of Steinberg groups which is exemplified in the following sections
by executing the envisioned programme for the universal Chevalley groups over the
integers.

\bigskip

The key example of our settings is the family of
special linear groups $\SL_n(\Z)$ (for $n \geqslant 3$),
generated by the set $S_n$ of elementary matrices
that differ from the identity by a single $\pm 1$ in an entry away from the diagonal.
With these generating sets one can think of $\SL_n(\Z)$ as
the universal Chevalley group over $\Z$ of type $\typeA{n-1}$.

Kazhdan constants for this family have been particularly well studied.
The first (partial) lower bounds of $\SL_3(\Z)$
was produced by Burger~\cite{Burger1991},
and eight years later Shalom~\cite{Shalom1999} gave lower bounds
for all groups $\SL_n(\Z)$ using the concept of bounded generation.
Subsequently, Kassabov~\cite{Kassabov2005} obtained a much better estimate which
we know to be asymptotically tight, as witnessed by the upper bound $\sqrt{\frac{2}{n}}$ by \.Zuk.
However, even for small $n$ these lower bounds differ by two orders of magnitude
from the upper bound.

A completely new method of establishing lower bounds for Kazhdan constants of
finitely generated groups emerged from the work of Ozawa~\cite{Ozawa2016}.
Using a quantitative version due to Netzer--Thom~\cite{NetzerThom2015},
one can utilise a computer to calculate such lower bounds for specific groups.
This was done by Netzer--Thom and Fujiwara--Kabaya~\cite{FujiwaraKabaya2017}
who obtained vastly improved bounds for $\SL_3(\Z)$ and $\SL_4(\Z)$.
This very direct approach culminated with the paper of the first author with Nowak~\cite{KalubaNowak2018},
in which the best known lower bounds for $\SL_n(\Z)$ with $n \in \{3,4,5\}$ were computed.
Finally, using a derived computer-assisted calculation and a form of induction,
the authors together with Nowak \cite{Kalubaetal2021} computed the best known bounds
for $\SL_{n+1}(\Z)$ with $n \geqslant 5$.
We list them in \cref{table intro}.
What is worth pointing out is that these results show that the actual value of
Kazhdan constants for special linear groups seems to be in the upper half of
the upper bound.

\begin{table}
    \begin{tabular}{@{}lcrrr@{}}\toprule
                               &                & \multicolumn{2}{c}{lower bound for $\kappa$}                                                                                                                              \\\cmidrule(){3-4}
        type                   & $n$            & $R=2$                                           & $R=3$                                                                 & comments                                        \\\midrule
        $\typeA{2}$            &                & $\Kzhconst{\lambdaSlthreea}{12}$                & $\Kzhconst{\lambdaSlthree}{12}$                                       & direct computations \cite{Kalubaetal2019}                            \\
        $\typeA{3}$            &                & $\Kzhconst{\lambdaSlfour}{24}$                                                                                                                                            \\
        $\typeA{4}$            &                & $\Kzhconst{\lambdaSlfive}{40}$                                                                                                                                            \\
        $\typeA{n}$            & $n\geqslant 5$ & $\sqrt{\frac{ 0.5(n-1)}{n(n+1)}}$               &                                                                       & see Section 5.1 of \cite{Kalubaetal2021}        \\
        $\typeD{n}$            & $n\geqslant 4$ &                                                 & $\sqrt{\frac{ \lambdaAdjAbc(n-2)}{n(n-1)}}$                           & see \cref{simply laced}                         \\
        $\typeE{6}$            &                &                                                 & $\Kzhconst{\fpeval{10*\lambdaAdjAbc}}{\fpeval{2*72}}$                                                                   \\
        $\typeE{7}$            &                &                                                 & $\Kzhconst{\fpeval{16*\lambdaAdjAbc}}{\fpeval{2*126}}$                                                                  \\
        $\typeE{8}$            &                &                                                 & $\Kzhconst{\fpeval{28*\lambdaAdjAbc}}{\fpeval{2*240}}$                                                                  \\
        \midrule
        $\typeB{2}=\typeC{2}$  &                & $\Kzhconst{\lambdaSpfourb}{16}$                 & $\Kzhconst{\lambdaSpfourc}{16}$                                       & direct computations                             \\
        $\typeB{n}, \typeC{n}$ & $n\geqslant 3$ & $\sqrt{\frac{\fpeval{\lambdaLevels/2.0}}{n^2}}$ & $\sqrt{\frac{\fpeval{\lambdaAdjCbc/2.0}(n-1)}{n^2}}$                  & see \cref{doubly laced via levels,doubly laced} \\
        $\typeF{4}$            &                &                                                 & $\Kzhconst{\fpeval{4*\lambdaAdjAbc + 3\lambdaAdjCbc}}{\fpeval{2*48}}$                                                   \\
        \midrule
        $\typeG{2}$            &                & $\Kzhconst{\lambdaGbb}{24}$                     &                                                                       & see \cref{triply laced}                         \\\bottomrule
        \addlinespace[\belowrulesep]
    \end{tabular}
    \caption{
        The best known lower bounds for Kazhdan constants.
        Note: for type $\typeC{n}$ the estimate with
        $R=2$ (via \cref{doubly laced via levels}) is better than with
        $R=3$ (via \cref{doubly laced}) for $n \leqslant 9$.
    }
    \label{table intro}
\end{table}

In this work we show that the somewhat \emph{ad hoc} constructions of
\cite{Kalubaetal2021} does in fact follow from a much more general construction
and applies to groups graded by root systems, as defined below.
We explore, in a quantitative and systematic way, the connection of
various group ring elements defined by the root systems of the underlying
Chevalley groups and we uncover a general method of proving property $T$
of which \cite{Kalubaetal2021} is a special case.
As an outcome, we are able to give the very first explicit asymptotically tight lower bounds for
universal Chevalley groups over $\Z$ corresponding to
all irreducible root systems of rank at least $2$.
In particular, we give concrete lower bounds for Kazhdan constants of
symplectic groups $\Sp_{2n}(\Z)$, with respect to the usual (Steinberg) generators. This was previously done by Neuhauser \cite{Neuhauser2003}, however his bounds are asymptotic to $\frac 1 n$, rather than the optimal $\frac 1 {\sqrt{n}}$.
This result complements therefore \cite[Theorem 7.12]{Ershovetal2017} which
gives the asymptotics of such lower bounds without providing concrete values.

\begin{thmx}
    Let $G$ be an elementary Chevalley group of type $\Omega$ over the integers,
    and let $S$ denote the set of its Steinberg generators.
    When $\Omega$ is irreducible and of rank at least $2$,
    then the pair $(G,S)$ has property $T$ with Kazhdan constant bounded below
    by the number indicated in \cref{table intro}.
\end{thmx}

 The parameter $R$ of \cref{table intro}  is the radius of a ball in the Cayley graph supporting the elements in the group ring that are used in a sum-of-Hermitian-squares decomposition, see \cref{sos def}.

It is very curious that the elements in the group ring mentioned above turn out to
satisfy a form of spectral gap in all the cases considered above,
but we are not able to prove it without a computer
even in the simplest case of $\Omega = \typeA{2}$.
The only computer-free result of relevance here is contained in a recent article
of Ozawa~\cite{Ozawa2022}.

Part of the motivation for looking into the symplectic groups was
the rough analogy between special linear groups over $\Z$ and
automorphism groups of free groups on one side, and
symplectic groups over $\Z$ and mapping class groups of surfaces on the other.
To have a hope of employing a similar line of attack on the question of
property T for mapping class groups one needs to first find an appropriate ``Steinberg'' generators.
The authors had constructed such a generating set, however were unable to prove the requred positivity statements and therefore our attempts eventually failed.

\subsection*{Related work}

A very basic idea behind this work lies in encoding the combinatorics of Steinberg generators in a root system.
In this spirit, Ershov--Kassabov--Jaikin-Zapirain~\cite{Ershovetal2017}
introduced groups graded by root systems by imposing a condition binding the
group structure with the additive structure of the root system.
The gradings we introduce are somewhat weaker, even though
(in case of arithmetic groups investigated here)
they do implicitly satisfy the conditions of \cite{Ershovetal2017}.

On the other hand, the induced ``non-commutative'' grading for
$\SAut(F_n)$
does not fall into the framework of \cite{Ershovetal2017}, and yet our definition
is sufficient to successfully apply our results (\cref{main thm}) to the group.
Explicitly, $\SAut(F_n)$ is generated by the set of \emph{Nielsen transformations}, or \emph{transvections}. These automorphisms multiply one of the generators by another one, either on the right or on the left. Their image in $\SL_n(\Z)$ is an elementary matrix, one of the Chevalley generators, and such matrices are naturally assigned to roots. This way we obtain an assignment sending transvections to roots -- this is the ``non-commutative grading'' alluded to above. If we pick a root, say $\alpha$, then the matrices in $\SL_n(\Z)$ assigned to $\pm \alpha$ generate a copy of $\SL_2(\Z)$, but this is no longer true for transvections associated to $\pm \alpha$, as they generate a copy of $\SAut(F_2)$. Moreover,  consider $F_3 = F(a,b,c)$, and  the  transvections
\[
\rho_{12} \colon \left\{ \begin{array}{ccl}
	a & & ab \\
	b & \mapsto & b \\
	c & & c
	\end{array} \right. \quad
\rho_{23} \colon \left\{ \begin{array}{ccl}
	a & & a \\
	b & \mapsto & bc \\
	c & & c \quad .
\end{array} \right.
\]
Their commutator $[\rho_{12}, \rho_{23}] = \rho_{12}\rho_{23}{\rho_{12}}^{-1}{\rho_{23}}^{-1}$ does not lie in the subgroup generated by transvections of the free factor $F(a,c)$. This violates the definition of a grading as introduced in \cite{Ershovetal2017}, and the problem persists for all $n>2$. Nevertheless,  $\SAut(F_n)$ for $n >3$ does have property $T$, and for $n>5$ it is proved using a technique that can be expressed in the same terms as in the proof of \cref{doubly laced via levels}: in \cite{Kalubaetal2021}, we introduced elements $\mathrm{Sq}_n$, $\mathrm{Adj}_n$, and $\mathrm{Op}_n$ lying in the group ring of $\SAut(F_n)$; these elements correspond to $\lev_2^n$, $\lev_3^n$, and $\lev_4^n$ respectively.  
 
We use gradings to establish property $T$
in a way that is very different to that used by
Ershov--Kassabov--Jaikin-Zapirain. They require that
the generators corresponding to orthogonal roots commute,
in order to ensure that certain subspaces of a Hilbert space
on which they act are orthogonal as well.
For us, orthogonality considerations are replaced by ensuring that
products of positive operators related to subspaces are positive.

\subsection*{Acknowledgements}
The authors would like to express their gratitude to Alain Valette for many interesting conversations, and the referees for helpful comments.
This work has received funding from the European Research Council (ERC) under the European Union's Horizon 2020 research and innovation programme (Grant agreement No. 850930).
The first author was supported by SPP 2026 “Geometry at infinity” funded by the Deutsche Forschungsgemeinschaft.

\section{Roots and sums of Hermitian squares}

\subsection{Ozawa's theorem}

Let $G$ be a discrete group generated by a symmetric generating set $S$.
The map
\[G \to G, \quad g \mapsto g^{-1}\]
extends linearly to a map $\R G \to \R G$ that we will denote by $x \mapsto x^*$.
Elements with $x^*=x$ will be called \emph{self adjoint}.

\begin{dfn}
	\label{sos def}
    We say that an element $x \in \R G$ is a \emph{sum of Hermitian squares} \iff there exist $\xi_1, \dots, \xi_n \in \R G$ with
    \[
        x = \sum_{i=1}^n {\xi_i}^*\xi_i.
    \]
    If additionally all the elements $\xi_i$ are supported in the ball of radius $R$
    with respect to the word metric on $G$ coming from $S$, we will say that
    the sum of squares decomposition is \emph{(witnessed) on radius $R$}
    and we will write $x \geqslant_R 0$.
\end{dfn}


\begin{thm}[{\cite{Ozawa2016}}]
    \label{Ozawa}
    Let $G$ be a group with a finite symmetric generating set $S$, and let
    \[
        \Delta = |S| - \sum_{s \in S} s \in \Z G
    \]
    denote the Laplacian. The group $G$ has property $T$ \iff there exists $\lambda>0$ such that
    \[
        \Delta^2 - \lambda \Delta
    \]
    admits a decomposition into a sum of Hermitian squares.
\end{thm}
If the decomposition is witnessed on radius $R$,
i.e. $\Delta^2 - \lambda \Delta \geqslant_R 0$,
we will say that the sum of Hermitian squares is a \emph{witness} of property $T$ \emph{of type} $(\lambda, R)$.
We remark that the number $\lambda$ gives a lower bound of
$\sqrt{\frac{2 \lambda }{|S|}}$
for the Kazhdan constant of $(G,S)$ -- see the proof of \cite[Proposition 5.4.5]{BekkaHarpeValette}.

\subsection{Root systems}

\begin{dfn}
    A subset $\Omega$ of a finite dimensional real vector space $V$ endowed with an inner product $\langle \cdot, \cdot \rangle$ is a \emph{root system} \iff for every $\alpha, \beta \in \Omega$ all of the following hold:
    \begin{enumerate}
        \item $0 \not\in \Omega$,
        \item $\Omega \cap \{ \lambda \alpha : \lambda \in \R \} = \{\pm \alpha\}$,
        \item $\beta - 2\frac{\langle \alpha, \beta \rangle}{\langle \alpha, \alpha \rangle}\alpha \in \Omega$,
        \item $2\frac{\langle \alpha, \beta \rangle}{\langle \alpha, \alpha \rangle} \in \Z$.
    \end{enumerate}

    Roots $\alpha, \beta \in \Omega$ are called \emph{proportional} (written $\alpha \sim \beta$) \iff they are related by homothety. If $\alpha$ and $\beta$ are not proportional, we call them \emph{non-proportional} and write $\alpha \not\sim \beta$.
    The \emph{rank} of $\Omega$ denotes the dimension of its linear span.
    The root system is \emph{reducible} \iff it is the disjoint union of two root systems that are orthogonal with respect to the inner product on $V$; otherwise, the system is \emph{irreducible}.
\end{dfn}

The irreducible root systems are classified by Dynkin diagrams, and hence come in the following types: $\typeA{n}$ with $n \geqslant 1$; $\typeB{n}$; with $n\geqslant 2$; $\typeC{n}$ with $n \geqslant 3$; $\typeD{n}$ with $n \geqslant 4$; $\typeE{n}$ with $n \in \{6,7,8\}$; $\typeF{4}$; and $\typeG{2}$. Types $\typeA{n}, \typeD{n}$ and $\typeE{n}$ are called \emph{simply laced}.

Since we insist on item $(2)$ above, we are only considering reduced root systems.

Every root system $\Omega$ has the associated Weyl group $W_\Omega$ acting on the roots. When the system is irreducible, the action is transitive on all roots of the same length. In particular, in the simply laced case, the Weyl group acts transitively.

\subsection{Decomposing the Laplacian}

Let $G$ be a finitely generated group, and $\Omega$ a root system.
To every root $\alpha \in \Omega$ we associate a finite symmetric subset $S_\alpha \subseteq G$ such that $G$ is generated by the finite set $S = \bigsqcup_{\alpha \in \Omega} S_\alpha$, the disjoint union of the sets $S_\alpha$.

We do not impose any further conditions on the subsets $S_\alpha$ and their interactions, and so the observations we will make in this section will be very general. It is however worth keeping in mind that in the subsequent sections we will look at groups where the subsets $S_\alpha$ and the subgroups they generate have specific properties. We will be mainly interested in the case where $\langle S_\alpha \rangle \cong \SL_2(\Z)$ for every $\alpha$, and where the elements of $S_\alpha$ and $S_\beta$ commute when $\alpha$ and $\beta$ are orthogonal.

\begin{dfn}[Laplacians for subgroups]
    For every $\alpha \in \Omega$ we set
    \[
        \Delta_\alpha = \vert S_\alpha \vert - \sum_{s \in S_\alpha} s \in \Z G
    \]
    and call it a \emph{root Laplacian}.
    For every subspace $W \leqslant V$ we set $\Omega_W = \Omega \cap W$, and
    \[
        \Delta_W = \sum_{\alpha \in \Omega_W} \Delta_\alpha.
    \]
\end{dfn}

In particular, $\Delta_V$ is the usual group Laplacian $\Delta$ associated to $G$ with respect to $S$. 
Since the sets $S_\alpha$ are assumed to be symmetric, the element $\Delta_\alpha$ is self adjoint for every $\alpha \in \Omega$, and hence every $\Delta_W$ is also self adjoint.

\begin{dfn}
    Let $W$ be a linear subspace of $V$. 
    %
    We  set $G_W$ to be the subgroup of $G$ generated by $S_W = \bigsqcup_{\alpha \in \Omega_W} S_\alpha$. We define the \emph{adjacency element} in $\Z G_W$ as follows:
    \[
        \Adj_W = \sum_{\alpha \in \Omega_W} \Delta_\alpha \left( \sum_{\alpha \not\sim \beta \in \Omega_W} \Delta_\beta \right) 
    \]
    with the convention that the empty sum is $0$.

    We will say that $W$ is
    \begin{itemize}
        \item \emph{irreducible} when $\Omega_W = \Omega \cap W$ is irreducible as a root system;
        \item \emph{admissible} when $\dim W = 2$ and $\Omega_W$ contains at least two roots which are non-proportional (and so, in particular, non-opposite).
    \end{itemize}
    The set of all 
    admissible subspaces will be denoted by $\A$ and the set of irreducible and admissible subspaces will be denoted by $\Airr$.
    Note that $\A$ is finite, since there are only finitely many roots, and every two non-proportional roots span a $2$-dimensional subspace.
\end{dfn}

\begin{thm}
    \label{main thm}
    Let $G$ be a group generated by a finite set $S$ decomposing as a disjoint union $\bigsqcup_{\alpha \in \Omega} S_\alpha=S$, where $\Omega$ is a root system.
    Suppose that $\dim V \geqslant 2$, that $V$ is the span of subspaces in $\A$,
    and that for every $W \in \A$
    \[
        \Adj_W - \lambda_W \Delta_W \geqslant_{R_W} 0
    \]
    for some $\lambda_W \geqslant 0$ and $R_W \geqslant 1$.
    Suppose additionally that for every $\alpha \in \Omega$
    there exists at least one $W \in \A$ such that
    $\alpha \in W$ and $\lambda_W > 0$.
    Then $(G,S)$ has property $T$ with a witness of type $(\lambda, R)$, where
    \[
        \lambda = \min \left\{
        \sum_{\alpha \in W \in \A} \lambda_W : {\alpha \in \Omega} \right\} > 0
        \quad \text{and} \quad R = \max_{W\in \A} R_W.
    \]
\end{thm}
\begin{proof}
    %
    Let $\Delta = \Delta_V$ be the Laplacian of $G$ with respect to $S$.
    We have
    \[
        \Delta^2 = \left( \sum_{\alpha \in \Omega} \Delta_\alpha \right)^2 = \Sq_V + \Adj_V
    \]
    where $\Sq_V = \sum_{\alpha \in \Omega} {\Delta_\alpha}^2$.
    Since every $\Delta_\alpha$ is self adjoint,
    $\Sq_V \geqslant_1 0$ is already a sum of squares witnessed on radius $1$.
    Thus, it is sufficient to show that
    \[
        \Adj_V - \lambda \Delta \geqslant_R 0.
    \]
    Since $V$ is spanned by admissible subspaces we note that
    $\Adj_V$ is a sum of $\Adj_W$ taken over admissible subspaces,
    and $\Delta = \sum_\alpha \Delta_\alpha$ is a sum of root Laplacians.
    Explicitly, we have for $\Adj_V$
    \begin{align*}
        \sum_{W \in \A} \Adj_W
         & =  \sum_{W \in \A} \left( \sum_{\alpha \in \Omega_W} \Delta_\alpha \left( \sum_{\alpha \not\sim \beta \in \Omega_W} \Delta_\beta \right) \right) \\
         & = \sum_{\alpha \in \Omega} \Delta_\alpha \left( \sum_{\alpha \in W \in \A} \sum_{\alpha \not\sim \beta \in \Omega_W} \Delta_\beta \right)        \\
         & = \sum_{\alpha \in \Omega} \Delta_\alpha \left( \sum_{\alpha \not\sim \beta \in \Omega} \Delta_\beta \right)                                     \\
         & = \Adj_V                                                                                                                                         \\
        \intertext{and for $\Delta=\Delta_V$}
        \sum_{W \in \A} \lambda_W \Delta_W
         & =  \sum_{W \in \A} \left( \lambda_W \sum_{\alpha \in \Omega_W} \Delta_\alpha \right)                                                             \\
         & = \sum_{\alpha \in \Omega} \Delta_\alpha \left( \sum_{\alpha \in W \in \A} \lambda_W \right)                                                     \\
         & = \sum_{\alpha \in \Omega} \lambda_\alpha \Delta_\alpha
    \end{align*}
    where $\lambda_\alpha = \sum_{\alpha \in W \in \A} \lambda_W > 0$ for every $\alpha$. Taking $ \lambda  = \min_{\alpha \in \Omega} \lambda_\alpha$ and combining these two equalities we conclude that
    \[
        \Adj_V - \lambda \Delta =
        \underbrace{\sum_{W \in \A} \left( \Adj_W - \lambda_W \Delta_W \right)}_{\geqslant_{R_W} 0} + \underbrace{\sum_{\alpha \in \Omega} (\lambda_\alpha - \lambda)\Delta_\alpha}_{\geqslant_1 0}
    \]
    as required. We deduce that $G$ has property $T$ from \cref{Ozawa}.
\end{proof}

The above theorem shows that, in order to prove property $T$ for $G$, it is enough to understand $\Adj_W$ for admissible subspaces $W$.

\section{Chevalley groups}

In this section we look at Chevalley groups, as defined by Steinberg in \cite{Steinberg1968}*{Section 3}. Since we will not need the precise definition, let us only sketch it here.
Let $\mathcal L$ denote a semi-simple Lie algebra over $\C$ associated to a root system $\Omega$. 
For every $\alpha \in \Omega$ one can choose a distinguished element $X_\alpha \in \mathcal L$. 
The choice of the elements $X_\alpha$ is unique up to signs and  automorphisms of $\mathcal L$.

Now take a faithful finite-dimensional $\C$-linear representation of $\mathcal L$ on a complex vector space $U$. 
We  define an automorphism $x_\alpha(t)  = \exp (tX_\alpha)$ of $U$. A subgroup of $\GL(U)$ generated by all the elements $x_\alpha(t)$ with $\alpha \in \Omega$ and $t \in \C$ is called a \emph{Chevalley group} of type $\Omega$.

In general, for a single $\Omega$ one gets more than one Chevalley group, depending on the representation $U$. There are however only finitely many Chevalley groups  of type $\Omega$, and among them there exists a \emph{universal} one that maps onto all the others, via a map that restricts to the identity on the Steinberg generators $x_\alpha(t)$.

\begin{dfn}[Elementary Chevalley group over $\Z$]
    In the above setup, we will refer to a subgroup of $\GL(U)$ generated by
    \[S = \{ x_\alpha(\pm1) : \alpha \in \Omega \}\]
    as an \emph{elementary Chevalley group} over $\Z$ of type $\Omega$. As above, we also have a \emph{universal} such group, with the same property. The set $S$ is the set of \emph{Steinberg generators}.

    In this definition $S$ comes with a natural map $S \to \Omega$, $x_\alpha(\pm1) \mapsto \alpha$ which we call a \emph{grading} of $(\langle S \rangle, S)$ by $\Omega$.
\end{dfn}

In the universal case, the groups above are isomorphic to the universal Chevalley
groups over $\Z$ that one obtains from the Chevalley--Demazure group scheme.

Given a Chevalley group over $\Z$, we set $S_\alpha = \{x_\alpha(\pm 1)\}$.
This allows us to use the notation and results of the previous section.

Note that an embedding of a root system $\Omega$ into a root system $\Omega'$
gives an embedding of the associated Lie algebras $\mathcal L_\Omega$ and $\mathcal L_{\Omega'}$ -- this can be seen by considering presentations of the Lie algebras, as given by Steinberg. The tricky aspect is that the presentations involve \emph{structure constants} $N_{\alpha,\beta}$, where $\alpha$ and $\beta$ range over all roots, and these structure constants cannot be read of the root system. Starting from a presentation for the $\mathcal L_{\Omega'}$, we obtain structure constants $N'_{\alpha,\beta}$. We can now remember these structure constants for all \emph{extraspecial pairs} of roots in $\Omega$, and obtain structure constants $N_{\alpha,\beta}$ for all pairs of roots in $\Omega$ from them, see \cite{Carter1972}*{Proposition 4.2.2}. In principle, these structure constants $N_{\alpha,\beta}$ may not agree with the constants $N'_{\alpha,\beta}$. The discussion after \cite{Carter1972}*{Proposition 4.2.2} tells us however that the constants $N_{\alpha,\beta}$ are uniquely determined by, and can be computed using a number of rules from the constants for extraspecial pairs. The constants $N'_{\alpha,\beta}$ also satisfy these rules, since they satisfy them over $\Omega'$. Hence, there exists a choice of a Chevalley basis for $\mathcal L_\Omega$ such that the map between the Chevalley bases of $\mathcal L_\Omega$ and $\mathcal L_{\Omega'}$ induces a map between Lie algebras. It is easy to see that this map is injective as a map of vector spaces. 

The upshot is that the representation $U$ used to define Chevalley groups of type $\Omega'$
works also for $\Omega$, and therefore a subgroup of a Chevalley group
of type $\Omega'$ generated by $x_\alpha(t)$ with $\alpha \in \Omega$ is itself
a Chevalley group of type $\Omega$.
The same argument works for elementary Chevalley groups over $\Z$.

\begin{ex}
    There are two explicit examples of immediate interest.
    \begin{enumerate}
        \item The universal Chevalley group of type $\typeA n$ over $\C$ is $\SL_{n+1}(\C)$.
              The corresponding Chevalley group over $\Z$ is $\SL_{n+1}(\Z)$.
              We may realise the system $\typeA{n}$ inside $\R^{n+1}$ endowed with the standard basis
              $\{e_i : 1 \leqslant i \leqslant n+1\}$ as the set
              \[
                  \{ \pm (e_i - e_j) : 1 \leqslant i < j \leqslant n+1\}.
              \]
              We then have $x_{\pm(e_i - e_j)}$ equal to $\I \pm \delta_{i,j}$,
              where $\delta_{i,j}$ differs form the zero matrix by a single $1$ in position $(i,j)$.

        \item Similarly, the universal Chevalley group of type $\typeC n$ is the symplectic group $\Sp_{2n}(\C)$, with the corresponding Chevalley group over $\Z$ being $\Sp_{2n}(\Z)$.
              Let us describe the matrices $x_\alpha$ in detail.

              The root system $\typeC n$ can be realised inside $\R^n$ as the set
              \[
                  \{ \pm e_i \pm e_j : 1 \leqslant i,j \leqslant n\} \s- \{0\}.
              \]
              We thus have long roots $\pm2 e_i$, with
              \[
                  x_{2 e_i} = \I + \delta_{i,i+n}, \quad x_{-2 e_i} = \I + \delta_{i+n,i},
              \]
              short roots of the form $e_i - e_j$ with $i \neq j$, corresponding to
              \[
                  x_{e_i - e_j} = \I +\delta_{i,j}  - \delta_{n+j,n+i},
              \]
              and finally short roots of the form $\pm(e_i + e_j)$ with $i \neq j$, corresponding to
              \[
                  x_{e_i + e_j} = \I + \delta_{i,j+n} + \delta_{j,i+n}, \quad x_{-e_i - e_j} = \I + \delta_{i+n,j} + \delta_{j+n,i}.
              \]

    \end{enumerate}
\end{ex}

\subsection{Simply laced types}

Let us introduce the following number related to a root system.

\begin{dfn}
    The number $\gamma(\Omega)$ for a root system $\Omega$ is defined as
    \[
        \gamma(\Omega) = \min_{\alpha \in \Omega}\left|
        \left\{
        \lspan (\alpha, \beta): \alpha \not\sim \beta \in \Omega \text{ and }
        \langle \alpha, \beta \rangle \neq 0
        \right\}
        \right|,
    \]
    where $\lspan$ denotes the $\R$-linear span.
\end{dfn}

Equivalently, we may set $\gamma(\Omega) = \min_{\alpha \in \Omega}\left|
\left\{
W \in \Airr: \alpha  \in W
\right\}
\right|.$

If $\gamma(\Omega)=0$, then $\Omega$ contains a root $\alpha$
that is orthogonal to all other roots but $-\alpha$, and therefore
$\Omega = \{\pm \alpha \} \sqcup \Omega'$ for some root system $\Omega'$,
i.e. $\Omega$ is reducible.

\begin{ex}
Let us compute $\gamma(\Omega)$ when $\Omega$ is of type $\typeA{n}$. Since the Weyl group acts transitively on $\Omega$, we have
\[
\gamma(\Omega) = \left|
\left\{
W \in \Airr: \alpha  \in W
\right\}
\right|.
\]
for every $\alpha \in \Omega$. Fixing $\alpha = e_1 - e_2$, we see that it is precisely the roots of the form $\pm(e_1 - e_j)$  and $\pm(e_2 - e_j)$ with $j>2$ that are not proportional and not orthogonal to $\alpha$. There are in total $4(n-1)$ such roots. Now, each of these roots lies in a unique two-dimensional subspace containing $\alpha$, and each such subspace intersects $\Omega$ in a system of type $\typeA{2}$. Such systems have precisely $6$ roots, two of which are proportional to $\alpha$. This means that the number of such spaces is $\frac{4(n-1)}{6-2} = n-1$.
\end{ex}

The values of $\gamma$, listed in \ref{simply laced},  will be used when computing witnesses of property $T$ for groups of simply laced type in \ref{simply laced}.

Now we may state an easy strengthening of \cref{main thm}.

\begin{cor}
    \label{main cor}
    Let $\Omega$ be an irreducible root system of rank at least $2$ and
    suppose that for every irreducible and admissible subspace $W \in \Airr$
    \[ \Adj_{W} - \lambda_{W} \Delta_{W} \geqslant_{R_{W}} 0 \]
    for some $\lambda_{W} > 0$ and $R_{W} \geqslant 1$.
    Then $(G,S)$ has property $T$ witnessed by
    a sum of Hermitian squares decomposition of type $(\lambda,R)$, where
    \[
        \lambda = \gamma(\Omega) \min \left\{
        \lambda_{W} : W \in \Airr
        \right\}, \quad
        R = \max \left\{
        R_{W} : W \in \Airr
        \right\}.
    \]
\end{cor}

\begin{proof}
    We need to prove that for every $\alpha\in \Omega$ there exists $W \in \Airr$ containing it,
    and therefore that $V$ is spanned by all subspaces from $\Airr$.
    The remaining conclusions will follow directly from \cref{main thm}.

    Since $\Omega$ is of rank at least $2$, there always exist $\beta\in \Omega$
    such that $W_\beta = \lspan (\alpha, \beta)$ is admissible.
    If no such $W_\beta$ is irreducible, then $\alpha$ must be orthogonal to all roots
    in $\Omega$ (except $\pm\alpha$).
    This contradicts the irreducibility of $\Omega$,
    hence at least one of those spaces $W_\beta$ must be irreducible.
\end{proof}

\begin{rmk}
    Note that the irreducibility condition on $\Omega_W$ above is equivalent to
    $\vert \Omega_W \vert \geqslant 5$ which might be easier to compute.
    Indeed, admissible subspaces are spanned by at least two non-collinear
    roots and if $\Omega_W$ contains exactly $4$ roots it is necessarily reducible.
\end{rmk}

\begin{thm}[{\cite{Kalubaetal2021}}]
    \label{adj in sl}
    Let $G = \SL_3({\Z})$ be the universal Chevalley group over $\Z$
    of type $\typeA{2}$ endowed with the set of Steinberg generators $S$.
    Let $V$ denote the ambient vector space of the root system.
    Then
    \[\Adj_V -\lambda \Delta_V \geqslant_R 0\]
    whenever $(\lambda, R) \in \big\{  (\lambdaAdjAbb, 2), (\lambdaAdjAbc, 3) \big\}$.
\end{thm}

\begin{rmk}
    The constant given in \cite{Kalubaetal2021} differs slightly from what is above.
    However one can obtain these constants using the same methods,
    the same code and patience.
\end{rmk}

\begin{thm}
    \label{simply laced}
    Let $G$ be an elementary Chevalley group over $\Z$ associated to
    a root system $\Omega$ of type $\typeA{n}$ with $n \geqslant 2$,
    or of type $\typeD{n}$ with $n \geqslant 4$,
    or of type $\typeE{6}$, $\typeE{7}$, or $\typeE{8}$.
    Then $(G,S)$ has property $T$ with a witness of type $(\gamma(\Omega)\lambda,R)$,
    where $(\lambda, R) \in \big\{ (\lambdaAdjAbb, 2), (\lambdaAdjAbc, 3) \big\}$.
\end{thm}
\begin{proof}
    We will apply \cref{main cor}.
    First note that all of the root systems above are irreducible.
    Secondly, since all of these types are simply laced,
    the only two possible cardinalities of admissible $\Omega_W$ are $4$ and $6$.
    The statement of \cref{main cor} tells us that we need only worry about the latter possibility,
    so let us fix an admissible subspace $W$ of $V$ with $\vert \Omega_W \vert =6$.
    This means precisely that $\Omega_W$ is itself isomorphic to the root system $\typeA{2}$,
    and the group $G_W$ is an elementary Chevalley group over $\Z$ of type $\typeA{2}$,
    and hence a quotient of $\SL_3(\Z)$. By \cref{adj in sl},
    \[
        \Adj_W - \lambda \Delta_W \geqslant_R 0
    \]
    is a sum of Hermitian squares and the conclusion follows from \cref{main cor}.
\end{proof}

The explicit lower bounds for the Kazhdan constants that can be extracted from the above result are listed in \cref{table}.

\begin{table}
    \begin{tabular}{@{}lccr@{}}\toprule
        type        & $|S|$     & $\gamma$ & lower bound for $\kappa$                          \\\midrule
        $\typeA{n}$ & $2n(n+1)$ & $(n-1)$  & $\sqrt{\frac{ \lambda_{\typeA{2}}\cdot (n-1)}{n(n+1)}}$ \\
        $\typeD{n}$ & $4n(n-1)$ & $2(n-2)$ & $\sqrt{\frac{ \lambda_{\typeA{2}}\cdot (n-2)}{n(n-1)}}$ \\
        $\typeE{6}$ & $144$     & $10$     & $\Kzhconst{\fpeval{10*\lambdaAdjAbc}}{144}$       \\
        $\typeE{7}$ & $252$     & $16$     & $\Kzhconst{\fpeval{16*\lambdaAdjAbc}}{252}$       \\
        $\typeE{8}$ & $480$     & $28$     & $\Kzhconst{\fpeval{28*\lambdaAdjAbc}}{480}$       \\ \bottomrule
        \addlinespace[\belowrulesep]
    \end{tabular}
    \caption{
        Explicit values for lower bounds for Kazhdan constants according to \cref{simply laced}.
        Note that the values for the exceptional root systems are based on using $(\lambda_{\typeA{2}}, R) = (\lambdaAdjAbc, 3)$.
        For the type $\typeA{n}$ we obtain here the same bound as in \cite{Kalubaetal2021}.
    }
        \label{table}
\end{table}

\subsection{Types with double bonds}

In this section we concentrate on the universal Chevalley groups over $\Z$ of type $\typeC{n}$, that is, symplectic groups $\Sp_{2n}(\Z)$. The discussion will also give a new proof of property $T$ for the elementary Chevalley groups over $\Z$ of types $\typeB{n}$ and $\typeF{4}$. Since the focus here is really on symplectic groups, we will often write $\typeC{2}$ instead of $\typeB{2}$.

The first two results are analogous to \cref{adj in sl,simply laced}
and the first one is proved by a computer-assisted calculation as well.

\begin{thm}
    \label{adj in sp}
    Let $G = \Sp_4({\Z})$ be the universal Chevalley group over $\Z$ of type
    $\typeB{2} = \typeC{2}$, endowed with the set of Steinberg generators $S$.
    Let $V$ denote the ambient vector space of the root system. Then
    \[\Adj_V -\lambda \Delta_V \geqslant_R 0\]
    for $(\lambda, R) = (\lambdaAdjCbc, 3)$.
\end{thm}

\begin{thm}
    \label{doubly laced}
    Let $G$ be an elementary Chevalley group over $\Z$ associated to
    a root system $\Omega$ of type $\typeB{n}$ or $\typeC{n}$ with $n \geqslant 2$,
    or of type $\typeF{4}$.
    Then $(G,S)$ has property $T$ with a witness of type $(\lambda,3)$, where
    $\lambda = (n-1)\lambda_{\typeC{2}}$ for types $\typeB{n}$ and $\typeC{n}$, and
    $\lambda = 3\lambda_{\typeC{2}} + 4 \lambda_{\typeA{2}}$ for type $\typeF{4}$,
    with $\lambda_{\typeA{2}} = \lambdaAdjAbc$ and $\lambda_{\typeC{2}} = \lambdaAdjCbc$.
\end{thm}
\begin{proof}
	First let us give a quick argument for why in the cases under consideration there is a witness of type $(\lambda, 3)$.	
    For admissible subspaces $W$ only three different types of root systems appear
    as $\Omega_W$: $\typeA{1}\times \typeA{1}$, $\typeA{2}$, and $\typeC{2}$.
    By \cref{main cor} we need to only worry about the last two types.

    We take $\lambda_W = \lambda_{\typeA{2}} =  \lambdaAdjAbc$ for the second, and
    $\lambda_W = \lambda_{\typeC{2}} = \lambdaAdjCbc$ for the third.
    We now apply \cref{adj in sl,adj in sp} and conclude that
    $(G,S)$ has property $T$ with a witness of type $(\lambda,3)$.
    
    \smallskip
    In order to compute $\lambda$, we will use \cref{main thm}, again disregarding type $\typeA{1}\times \typeA{1}$, since it does not contribute.

    We will focus on $\typeC{n}$, since $\typeB{n}$ is its dual,
    and so the computation is literally the same.
    When $\Omega$ is of type $\typeC{n}$ and $\alpha$ is a long root,
    then we know that every  $W\in \Airr$
    containing $\alpha$ is of type $\typeC{2}$,
    and thus the corresponding $\lambda$ is $(n-1)\lambda_{\typeC{2}}$ -- the coefficient $n-1$ comes from counting $W \in \Airr$ of type $\typeC{2}$ containing $\alpha$.
    When $\alpha$ is a short root, there is a single $W \in \Airr$ of type $\typeC{2}$ that contains $\alpha$, and there are $2(n-2)$ such subspaces of type $\typeA{2}$; hence
    we obtain $\lambda$ equal to $\lambda_{\typeC{2}} + 2(n-2)\lambda_{\typeA{2}}$.
    The minimum of these two is $(n-1)\lambda_{\typeC{2}}$.

    Finally, consider $\Omega$ of type $\typeF{4}$.
    Since this type is self-dual,
    we may look at short roots without losing any generality, and
    since all roots of the same length form an orbit under the action of the Weyl group,
    we need to only carry out our computation for a single root of our choosing.
    By counting, we find that for every short root $\alpha \in \typeF{4}$
    there exist $18$ other, non-proportional roots $\beta$ such that
    $(\alpha, \beta)$ spans an admissible subspace of type $\typeC{2}$.
    Since the same subspace will be spanned for $6$ different choices of $\beta$,
    we see that there are precisely $3$ admissible subspaces containing $\beta$.
    By similar considerations we arrive at $16$ roots which give rise to $4$
    admissible subspaces $V$ such that $\Omega_V \cong \typeA{2}$.
    Hence we conclude that
    $\lambda = 3\lambda_{\typeC{2}} + 4\lambda_{\typeA{2}}$ for type $\typeF{4}$.
\end{proof}

For symplectic groups, with a more involved argument we will now
find a witness for property $T$ on radius $2$.
Our considerations will apply to $\Sp_{2n}({\Z})$ for $n\geqslant 3$;
for $n=2$ we offer the following direct computation.

\begin{thm}
    Let $G = \Sp_4({\Z})$ be the universal Chevalley group over $\Z$ of type
    $\typeB{2} = \typeC{2}$ endowed with the set of Steinberg generators $S$.
    The pair $(G,S)$ has property $T$ with a witness of type
    $(\lambda, R) \in \left\{(\lambdaSpfourb, 2), (\lambdaSpfourc, 3) \right\}$.
\end{thm}


Now let $\Omega$ denote the root system of type $\typeC{n}$ embedded in a vector space $V$, and let $G$ be the associated universal Chevalley group over $\Z$, that is $\Sp_{2n}({\Z})$, endowed with the Steinberg generators. Let $\Delta = \Delta_V$. We have
\[
    \Delta^2 =
    \sum_{\alpha, \beta \in \Omega} \Delta_\alpha \Delta_\beta
    = \sum_{\alpha \in \Omega} {\Delta_\alpha}^2 + \sum_{W \in \A} \Adj_W.
\]
Let $\Omega_\mathrm{long}$ denote the set of long roots in $\Omega$, and let $\Omega_\mathrm{short}$ denote the set of short roots.
We let
\[
    \Sq_\mathrm{long} = \sum_{\alpha \in \Omega_\mathrm{long}} {\Delta_\alpha}^2
    \quad \text{and} \quad
    \Sq_\mathrm{short} = \sum_{\alpha \in \Omega_\mathrm{short}} {\Delta_\alpha}^2.
\]

For a two-dimensional admissible subspace $W$, the root system $\Omega_W$ can be of type $\typeA{1}\times \typeA{1}$, $\typeA{1}\times \typeC{1}$, $\typeA{2}$, or $\typeC{2}$, where the distinction between types $\typeA{1}$ and $\typeC{1}$ is that we consider the former to consist of short roots in $\Omega$, and the latter to consist of long roots. Note that type ${\typeC{1}\times\typeC{1}}$ does not appear, since any plane containing two long roots will intersect $\Omega$ in a system of type $\typeC{2}$.

For $\type{Z}{}$ being one of the above types,
we let $\Adj_{\type{Z}{}}$ denote the sum of the elements $\Adj_W$
with $\Omega_W$ of type $\type{Z}{}$.

We now have
\[
    \Delta^2 =
    \Sq_\mathrm{long} +
    \Sq_\mathrm{short} +
    \Adj_{\typeC{2}} +
    \Adj_{\typeA{1}\times\typeC{1}} +
    \Adj_{\typeA{2}} +
    \Adj_{\typeA{1}\times\typeA{1}}.
\]
In the above equation, all the elements depend on $n$.
We have been suppressing this dependence so far, but we will now make it explicit.

\begin{dfn}[Levels]
    For every $n \geqslant 2$, we define four elements of $\Z \Sp_{2n}(\Z)$ as follows.
    \begin{align*}
        \lev_1^n & =   \Sq_\mathrm{long}                                 \\
        \lev_2^n & =   \Sq_\mathrm{short}  + \Adj_{\typeC{2}}            \\
        \lev_3^n & = \Adj_{\typeA{1}\times\typeC{1}}  + \Adj_{\typeA{2}} \\
        \lev_4^n & = \Adj_{\typeA{1}\times\typeA{1}}
    \end{align*}
\end{dfn}

The reason for calling these elements levels
comes from the way they behave under the action of the Weyl group.
More specifically, consider the standard embedding of
$\Sp_{2n}(\Z)$ into $\Sp_{2m}(\Z)$ for $m>n$.
Let $W_{\typeC{m}}$ denote the Weyl group of type $\typeC{m}$ and
denote by $x^w$ the action of $w\in W_{\typeC{m}}$ on an element $x$
from the corresponding group (ring).
It is clear from the definitions that if we take
$\Adj_{\type{Z}{}}\in \R \Sp_{2n}(\Z)$,
then the sum
\[\sum_{w \in W_{\typeC{m}}} {\Adj_{\type{Z}{}}}^w\]
is equal to $\Adj_{\type{Z}{}}$ formed in $\R \Sp_{2m}(\Z)$,
up to multiplication by a constant depending on $n,m$ and $\type{Z}{}$,
and similarly for $\Sq_\mathrm{long}$ and $\Sq_\mathrm{short}$.
We group these elements into the levels, precisely depending on these constants.

\begin{lem} Take $n \geqslant m \geqslant i$ and $i \in \{1,2,3,4\}$. We have
    \[
        \sum_{w \in W_{\typeC{n}}} \left(\lev_i^m\right)^w = \frac{2^n \cdot m! \cdot (n-i)!}{(m-i)!} \lev_i^n.
    \]
\end{lem}

A direct computation now gives us the following.

\begin{thm}
    \label{C3}
    Let $G = \Sp_6({\Z})$ be the universal Chevalley group over $\Z$ of type $\typeC{3}$
    endowed with the set of Steinberg generators $S$.
    Let $V$ denote the ambient vector space of the root system.
    Then for $\lambda = \lambdaLevels$ we have
    \[\lev_2^3 +  \lev_3^3 -\lambda \Delta_V \geqslant_2 0.\]
\end{thm}

\begin{thm}
    \label{doubly laced via levels}
    Let $G = \Sp_{2n}({\Z})$ be the universal Chevalley group over $\Z$
    of type $\typeC{n}$ for $n \geqslant 3$,
    endowed with the set of Steinberg generators $S$.
    The pair $(G,S)$ has property $T$ with a witness of type $( \lambdaLevels, 2)$.
\end{thm}
\begin{proof}
    Let $\Omega$ denote the root system of type $\typeC{n}$ embedded in a vector space $V$.
    Let $V_3$ be a $3$-dimensional subspace of $V$ such that
    $\Omega_{V_3}$ is a root system of type $\typeC{3}$;
    this way we have a fixed copy of $\Sp_6(\Z)$ embedded in $G$.
    By \cref{C3}
    \[ \lev_2^3 +  \lev_3^3 - \lambda \Delta_{V_3} \geqslant_2 0 \]
    for $\lambda = \lambdaLevels$. We see that
    \begin{multline*}
        0 \leqslant_2
        \frac{1}{3\cdot 2^n \cdot (n-3)!} \sum_{w \in W_{\typeC{n}}} \left(
        \lev_2^3 +  \lev_3^3 - \lambda \Delta_{V_3}
        \right)^w = \\
        = 2(n-2)\lev_2^n + 2\lev_3^n - \lambda (n-2)\left(
        (n-1) \sum_{\alpha \in \Omega_\mathrm{long}}
        \Delta_\alpha + 2 \sum_{\alpha \in \Omega_\mathrm{short}} \Delta_\alpha
        \right).
    \end{multline*}
    Since every $\Delta_\alpha \geqslant_1 0$, we conclude that
    \begin{equation*}
        \tag{$*$}\label{*}
        2(n-2)\lev_2^n +  2\lev_3^n - 2\lambda (n-2) \Delta \geqslant_2 0,
    \end{equation*}
    where $\Delta = \Delta_V$.

    Let $V_2$ be a two-dimensional subspace of $V$ such that $\Omega_{V_2}$ is a root system of type $\typeA{2}$. 
    By \cref{simply laced}, the element $\Adj_{V_2}$
    is a sum of Hermitian squares.
    This implies that so is the element $\Adj_{\typeA{2}}$ in the definition of $\lev_3^3$.
    By our definition of the grading, generating sets $S_\alpha$ and  $S_\beta$
    assigned to orthogonal roots commute with each other
    and therefore $\Adj_{\typeA{1} \times \typeC{1}}$ is a sum of Hermitian squares
    -- the details of this argument can be seen in the proof of \cite{Kalubaetal2021}*{Lemma~3.6}.
    We conclude that $\lev_3^3$, and hence $\lev_3^n$, are sums of Hermitian squares. Adding a positive multiple of  $\lev_3^n$ to \eqref{*} and dividing by a positive constant shows that
    \[
        \lev_2^n +  \lev_3^n - \lambda \Delta \geqslant_2 0.
    \]
    Since $\lev_1^n$ is blatantly a sum of Hermitian squares,
    and since the argument we just used for $\Adj_{\typeA{1} \times \typeC{1}}$ applies also to $\lev_4^n$, we finish by observing that
    \[
        \Delta^2 - \lambda \Delta =
        \sum_{i=1}^4 \lev_i^n -  \lambda \Delta \geqslant_2 0.
        \qedhere
    \]
\end{proof}

\subsection{The triple bond}

A direct computer calculation yields the following result.

\begin{thm}
    \label{triply laced}
    Let $G$ be the Chevalley group over $\Z$ of type $\typeG{2}$, and
    let $S$ be the set of its Steinberg generators.
    The pair $(G,S)$ has property $T$ with a witness of type $(\lambdaGbb,2)$.
\end{thm}

Let us also mention the following result, also obtained by a direct calculation.
No result of this paper makes use of it, however, for the sake of completeness
(i.e. just because we can), we show the element $\Adj$ of the last remaining type
behaves in a similar manner to all the previous ones.

\begin{thm}
    \label{adj in G2}
    Let $G$ be the Chevalley group over $\Z$ of type $\typeG{2}$, and
    let $S$ be the set of its Steinberg generators.
    Let $V$ denote the ambient vector space of the root system. Then
    \[\Adj_V -\lambda \Delta_V \geqslant_R 0\]
    for $(\lambda, R) = (\lambdaAdjGbc, 3)$.
\end{thm}

\section{Replication}

The code used to perform the computations has been stored
in the Zenodo repository \cite{Kaluba2023zenodo}.
A fair share of commentaries (both about computations and the
certification) is included in notebooks contained therein.

For all the gory details we direct the interested reader to study the code used
for replication of \cite{Kalubaetal2019,KalubaNowak2018}.

To pay tribute to the authors below we list a few significant pieces of software
that are used for our computations:

\begin{itemize}
    \item the code is written in julia programming language \cite{Bezanson2017julia};
    \item authors' \href{https://github.com/kalmarek/PropertyT.jl}{PropertyT.jl}
          is used to formulate the optimization problems through
          \href{https://jump.dev/}{JuMP.jl} \cite{Lubin2023} package for
          mathematical programming;
    \item \href{https://github.com/cvxgrp/scs}{Splitting conic solver (scs)}
          \cite{ODonoghue2016} and
          \href{https://oxfordcontrol.github.io/COSMO.jl/stable/}{Conic operator splitting method (COSMO)}
          \cite{Garstka2021} solvers are used to solve the problems;
    \item \href{https://github.com/JuliaIntervals/IntervalArithmetic.jl}{IntervalArithmetic.jl} package
          \cite{Sanders2023intervals} is used to certify the soundness of
          computations in floating point arithmetic.
\end{itemize}
Finally the authors' own packages \href{https://github.com/kalmarek/Groups.jl}{Groups.jl},
\href{https://github.com/kalmarek/StarAlgebas.jl}{StarAlgebras.jl}, and
\href{https://github.com/kalmarek/SymbolicWedderburn.jl}{SymbolicWedderburn.jl}
provide the modelling of the algebraic structures.
Thanks especially to the last package, all computational statements with $R = 2$
can be reproven on an average desktop computer within minutes.
Same computations with $R = 3$ though require substantial computational resources
and patience.

\bibliographystyle{amsalpha}
\bibliography{bibliography}

\begin{bibdiv}
\begin{biblist}

\bib{BacherDelaharpe1994}{article}{
      author={Bacher, R.},
      author={De~La~Harpe, P.},
       title={Exact {V}alues of {K}azhdan {C}onstants for {S}ome {F}inite
  {G}roups},
        date={1994-01},
        ISSN={0021-8693},
     journal={Journal of Algebra},
      volume={163},
      number={2},
       pages={495\ndash 515},
}

\bib{BekkaHarpeValette}{book}{
      author={Bekka, Bachir},
      author={de~la Harpe, Pierre},
      author={Valette, Alain},
       title={Kazhdan's property ({T})},
      series={New Mathematical Monographs},
   publisher={Cambridge University Press, Cambridge},
        date={2008},
      volume={11},
        ISBN={978-0-521-88720-5},
         url={https://doi.org/10.1017/CBO9780511542749},
      review={\MR{2415834}},
}

\bib{Bezanson2017julia}{article}{
      author={Bezanson, Jeff},
      author={Edelman, Alan},
      author={Karpinski, Stefan},
      author={Shah, Viral~B.},
       title={Julia: A fresh approach to numerical computing},
        date={2017},
     journal={SIAM Review},
      volume={59},
      number={1},
       pages={65\ndash 98},
      eprint={https://doi.org/10.1137/141000671},
         url={https://doi.org/10.1137/141000671},
}

\bib{BorelHarish-Chandra1962}{article}{
      author={Borel, Armand},
      author={Harish-Chandra},
       title={Arithmetic subgroups of algebraic groups},
        date={1962},
        ISSN={0003-486X},
     journal={Annals of Mathematics. Second Series},
      volume={75},
       pages={485\ndash 535},
      review={\MR{147566}},
}

\bib{Burger1991}{article}{
      author={Burger, M.},
       title={Kazhdan constants for {${\rm SL}(3,{\bf Z})$}},
        date={1991},
        ISSN={0075-4102},
     journal={Journal f\"{u}r die Reine und Angewandte Mathematik. [Crelle's
  Journal]},
      volume={413},
       pages={36\ndash 67},
      review={\MR{1089795}},
}

\bib{Carter1972}{book}{
      author={Carter, Roger~W.},
       title={Simple groups of lie type},
      series={A @Wiley-Interscience publication},
   publisher={Wiley},
     address={London [u.a.]},
        date={1972},
        ISBN={0471137359},
}

\bib{Cartwright1994}{article}{
      author={Cartwright, Donald~I.},
      author={M{\l}otkowski, Wojciech},
      author={Steger, Tim},
       title={{Property (T) and $\widetilde{A}_2$ groups}},
    language={en},
        date={1994},
     journal={Annales de l'Institut Fourier},
      volume={44},
      number={1},
       pages={213\ndash 248},
         url={http://www.numdam.org/articles/10.5802/aif.1395/},
      review={\MR{95j:20024}},
}

\bib{DymaraJanuszkiewicz2002}{article}{
      author={Dymara, Jan},
      author={Januszkiewicz, Tadeusz},
       title={Cohomology of buildings and their automorphism groups},
        date={2002},
        ISSN={0020-9910,1432-1297},
     journal={Invent. Math.},
      volume={150},
      number={3},
       pages={579\ndash 627},
         url={https://doi.org/10.1007/s00222-002-0242-y},
      review={\MR{1946553}},
}

\bib{delaHarpeValette1989}{article}{
      author={de~la Harpe, Pierre},
      author={Valette, Alain},
       title={La propri\'{e}t\'{e} {$(T)$} de {K}azhdan pour les groupes
  localement compacts (avec un appendice de {M}arc {B}urger)},
        date={1989},
        ISSN={0303-1179},
     journal={Ast\'{e}risque},
      number={175},
       pages={158},
        note={With an appendix by M. Burger},
      review={\MR{1023471}},
}

\bib{ErshovJaikin2010}{article}{
      author={Ershov, Mikhail},
      author={Jaikin-Zapirain, Andrei},
       title={Property ({T}) for noncommutative universal lattices},
        date={2010},
        ISSN={0020-9910,1432-1297},
     journal={Invent. Math.},
      volume={179},
      number={2},
       pages={303\ndash 347},
         url={https://doi.org/10.1007/s00222-009-0218-2},
      review={\MR{2570119}},
}

\bib{Ershovetal2017}{article}{
      author={Ershov, Mikhail},
      author={Jaikin-Zapirain, Andrei},
      author={Kassabov, Martin},
       title={Property {(T)} for groups graded by root systems},
        date={2017},
        ISSN={0065-9266},
     journal={Memoirs of the American Mathematical Society},
      volume={249},
      number={1186},
       pages={v+135},
      review={\MR{3724373}},
}

\bib{FujiwaraKabaya2017}{article}{
      author={Fujiwara, Koji},
      author={Kabaya, Yuichi},
       title={Computing {K}azhdan constants by semidefinite programming},
        date={2019},
     journal={Experimental Mathematics},
      volume={28},
      number={3},
       pages={301\ndash 312},
      eprint={https://doi.org/10.1080/10586458.2017.1396509},
         url={https://doi.org/10.1080/10586458.2017.1396509},
}

\bib{Garstka2021}{article}{
      author={Garstka, Michael},
      author={Cannon, Mark},
      author={Goulart, Paul},
       title={{COSMO}: A conic operator splitting method for convex conic
  problems},
        date={2021},
     journal={Journal of Optimization Theory and Applications},
         url={https://doi.org/10.1007/s10957-021-01896-x},
}

\bib{Kazhdan1967}{article}{
      author={Ka\v{z}dan, D.~A.},
       title={On the connection of the dual space of a group with the structure
  of its closed subgroups},
        date={1967},
        ISSN={0374-1990},
     journal={Akademija Nauk SSSR. Funkcional\cprime nyi Analiz i ego
  Prilo\v{z}enija},
      volume={1},
       pages={71\ndash 74},
      review={\MR{0209390}},
}

\bib{Kassabov2005}{article}{
      author={Kassabov, Martin},
       title={Kazhdan constants for {${\rm SL}_n({\mathbb{Z}})$}},
        date={2005},
        ISSN={0218-1967},
     journal={Internat. J. Algebra Comput.},
      volume={15},
      number={5-6},
       pages={971\ndash 995},
         url={https://doi.org/10.1142/S0218196705002712},
      review={\MR{2197816}},
}

\bib{Kaluba2023zenodo}{misc}{
      author={Kaluba, Marek},
      author={Kielak, Dawid},
       title={Replication software for 2306.12358},
   publisher={Zenodo},
        date={2023},
         url={https://doi.org/10.5281/zenodo.8094798},
}

\bib{Kalubaetal2021}{article}{
      author={Kaluba, Marek},
      author={Kielak, Dawid},
      author={Nowak, Piotr~W.},
       title={On property ({T}) for {$\mathrm{Aut}(F_n)$} and
  {$\mathrm{SL}_n(\mathbb {Z})$}},
        date={2021},
        ISSN={0003-486X},
     journal={Annals of Mathematics. Second Series},
      volume={193},
      number={2},
       pages={539\ndash 562},
      review={\MR{4224715}},
}

\bib{KalubaNowak2018}{article}{
      author={Kaluba, Marek},
      author={Nowak, Piotr~W.},
       title={Certifying numerical estimates of spectral gaps},
        date={2018},
        ISSN={1867-1144},
     journal={Groups Complex. Cryptol.},
      volume={10},
      number={1},
       pages={33\ndash 41},
         url={https://doi.org/10.1515/gcc-2018-0004},
      review={\MR{3794918}},
}

\bib{Kalubaetal2019}{article}{
      author={Kaluba, Marek},
      author={Nowak, Piotr~W.},
      author={Ozawa, Narutaka},
       title={{$\mathrm{Aut}(\mathbb{F}_5)$} has property {$(T)$}},
        date={2019aug},
     journal={Mathematische Annalen},
      volume={375},
      number={3-4},
       pages={1169\ndash 1191},
}

\bib{Lubin2023}{article}{
      author={Lubin, Miles},
      author={Dowson, Oscar},
      author={Garcia, Joaquim~Dias},
      author={Huchette, Joey},
      author={Legat, Beno{\^i}t},
      author={Vielma, Juan~Pablo},
       title={Jump 1.0: Recent improvements to a modeling language for
  mathematical optimization},
        date={2023},
     journal={Mathematical Programming Computation},
}

\bib{Neuhauser2003}{article}{
      author={Neuhauser, Markus},
       title={Kazhdan’s {P}roperty {$T$} for the {S}ymplectic {G}roup over a
  {R}ing},
        date={2003-09},
        ISSN={1370-1444},
     journal={Bulletin of the Belgian Mathematical Society - Simon Stevin},
      volume={10},
      number={4},
}

\bib{Nitsche2020}{article}{
      author={Nitsche, Martin},
       title={Computer proofs for property ({T}), and {SDP} duality},
        date={2020-09},
      eprint={2009.05134},
}

\bib{NetzerThom2015}{article}{
      author={Netzer, Tim},
      author={Thom, Andreas},
       title={{Kazhdan's Property (T) via Semidefinite Optimization}},
        date={2015jul},
        ISSN={1058-6458},
     journal={Experimental Mathematics},
      volume={24},
      number={3},
       pages={371\ndash 374},
  url={http://www.tandfonline.com/doi/full/10.1080/10586458.2014.999149},
}

\bib{ODonoghue2016}{article}{
      author={O'Donoghue, Brendan},
      author={Chu, Eric},
      author={Parikh, Neal},
      author={Boyd, Stephen},
       title={Conic optimization via operator splitting and homogeneous
  self-dual embedding},
        date={2016June},
     journal={Journal of Optimization Theory and Applications},
      volume={169},
      number={3},
       pages={1042\ndash 1068},
         url={http://stanford.edu/~boyd/papers/scs.html},
}

\bib{Ozawa2016}{article}{
      author={Ozawa, Narutaka},
       title={Noncommutative real algebraic geometry of {K}azhdan's property
  ({T})},
        date={2016},
        ISSN={1474-7480},
     journal={J. Inst. Math. Jussieu},
      volume={15},
      number={1},
       pages={85\ndash 90},
         url={https://doi.org/10.1017/S1474748014000309},
      review={\MR{3427595}},
}

\bib{Ozawa2022}{article}{
      author={Ozawa, Narutaka},
       title={A substitute for {K}azhdan's property ({T}) for universal
  non-lattices},
        date={2022-07},
      eprint={2207.05272},
}

\bib{Sanders2023intervals}{misc}{
      author={Sanders, David~P.},
      author={Benet, Luis},
      author={Ferranti, Luca},
      author={Agarwal, Krish},
      author={Richard, Benoît},
      author={Grawitter, Josua},
      author={Gupta, Eeshan},
      author={Forets, Marcelo},
      author={Herbst, Michael~F.},
      author={yashrajgupta},
      author={Hanson, Eric},
      author={van Dyk, Braam},
      author={Rackauckas, Christopher},
      author={Vasani, Rushabh},
      author={Miclu{\c{t}}a-Câmpeanu, Sebastian},
      author={Olver, Sheehan},
      author={Koolen, Twan},
      author={Wormell, Caroline},
      author={Karrasch, Daniel},
      author={Widmann, David},
      author={Vázquez, Favio~André},
      author={Dalle, Guillaume},
      author={Sarnoff, Jeffrey},
      author={TagBot, Julia},
      author={O'Bryant, Kevin},
      author={Carlsson, Kristoffer},
      author={Piibeleht, Morten},
      author={Giordano, Mosè},
      author={Zhang, Qi},
      author={Deits, Robin},
       title={Juliaintervals/intervalarithmetic.jl: v0.20.9},
   publisher={Zenodo},
        date={2023},
         url={https://doi.org/10.5281/zenodo.8037734},
}

\bib{Shalom1999}{article}{
      author={Shalom, Yehuda},
       title={Bounded generation and {K}azhdan's property ({T})},
        date={1999},
        ISSN={0073-8301},
     journal={Inst. Hautes \'{E}tudes Sci. Publ. Math.},
      number={90},
       pages={145\ndash 168 (2001)},
         url={http://www.numdam.org/item?id=PMIHES_1999__90__145_0},
      review={\MR{1813225}},
}

\bib{Steinberg1968}{book}{
      author={Steinberg, Robert},
       title={Lectures on {C}hevalley groups},
      series={University Lecture Series},
   publisher={American Mathematical Society, Providence, RI},
        date={1968},
      volume={66},
        ISBN={978-1-4704-3105-1},
        note={Notes prepared by John Faulkner and Robert Wilson, Revised and
  corrected edition of the 1968 original [ MR0466335], With a foreword by
  Robert R. Snapp},
      review={\MR{3616493}},
}

\bib{Vaserstein1968}{article}{
      author={Vaser\v{s}te\u{\i}n, L.~N.},
       title={Groups having the property {$T$}},
        date={1968},
        ISSN={0374-1990},
     journal={Akademija Nauk SSSR. Funkcional\cprime nyi Analiz i ego
  Prilo\v{z}enija},
      volume={2},
      number={2},
       pages={86},
      review={\MR{0228620}},
}

\bib{Zuk2003}{article}{
      author={{\.Z}uk, Andrzej},
       title={Property {$(T)$} and {K}azhdan constants for discrete groups},
        date={2003jun},
     journal={Geometric And Functional Analysis},
      volume={13},
      number={3},
       pages={643\ndash 670},
}

\end{biblist}
\end{bibdiv}

\end{document}